\documentclass[dvipdfmx, 11pt, a4paper]{article}
\usepackage{amsmath}
\usepackage{amssymb}
\usepackage{amsthm}
\usepackage{amssymb}
\usepackage{verbatim}
\usepackage[dvipdfmx]{graphicx}
\usepackage{color}    
\usepackage{hyperref}
\usepackage{stackrel}

\begin{document}

\title{Transfinite Version of Welter's Game}

\author{Tomoaki Abuku\footnote{affiliation:University of Tsukuba} \footnote{mail:buku3416@gmail.com}}

\theoremstyle{definition} 
\newtheorem{theorem}{Theorem}[section]
\newtheorem{definition}[theorem]{Definition}
\newtheorem{lemma}[theorem]{Lemma}
\newtheorem{proposition}[theorem]{Proposition}
\newtheorem{corollary}[theorem]{Corollary}
\newtheorem{example}[theorem]{Example}
\newcommand{\mex}{\mathrm{\mex}}
\renewcommand{\labelitemiii}{\cdot}
\renewcommand{\labelenumi}{(\theenumi)}
\newcounter{linenumber}
\newcommand{\bracket}[1]{\left[#1 \right]}

\maketitle

\begin{abstract}
We study the transfinite version of Welter's Game, a combinatorial game played on a belt divided into squares numbered with general ordinal.
In particular, we give a straight-forward solution for the transfinite version, based on those of the transfinite version of Nim and the original version of Welter's Game.
\end{abstract}



\section{Introduction}
\label{intro}
\subsection{Impartial game}
This paper discusses only ``impartial'' combinatorial games in normal form, that is games with the following characters:
\begin{itemize}
\item\ Two players alternately make a move.
\item\ No chance elements (the possible moves in any given position is determined in advance).
\item\ Both players have complete knowledge of the game states.
\item\ The game terminates in finitely many moves.
\item\ A player who makes the last move wins.
\item\ Both players have the same set of the possible moves in any position.
\end{itemize}

The original version of Nim and Welter's Game are ``short'' games (namely there are only a finite number of positions that can be reached from the initial position, and a position may never be repeated in a play).

\begin{definition}[outcome classes]
A game position is called an $\mathcal{N}$-position (resp. a $\mathcal{P}$-position) if the first player (resp. the second player) has a winning strategy.
\end{definition}

Clearly, all impartial game positions are classified into $\mathcal{N}$-positions or $\mathcal{P}$-positions.

\begin{theorem}[Bouton\cite{Bouton}]
If $G$ is an $\mathcal{N}$-position, there exists a move from $G$ to a $\mathcal{P}$-position.
If $G$ is a $\mathcal{P}$-position, there exists no move from $G$ to a $\mathcal{P}$-position.
\end{theorem}

\begin{definition}
Let $G$ and $G'$ be game positions. The notation $G \rightarrow G'$ means that $G'$ can be reached from $G$ by a single move.  
\end{definition}

\subsection{Nim and Grundy value}
Let us denote by $\mathbb{Z}$ the set of all integers and by $\mathbb{N}_{0}$ the set of all nonnegative integers. \\

Nim is a well-known impartial game with the following rules:
\begin{itemize}
\item\ It is played with several heaps of tokens.
\item\ The legal move is to remove any number of tokens (but necessarily at least one token) from any single heap.
\item\ The end position is the state of no heaps of tokens.
\end{itemize}

\begin{definition}[nim-sum]
The value obtained by adding numbers in binary form without carry is called nim-sum. The nim-sum of nonnegative integers $m_{1},\ldots,m_{n}$  is written as 
\begin{center}
$m_{1}\oplus \cdots \oplus m_{n}$.
\end{center}
\end{definition}

The set $\mathbb{N}_{0}$ is isomorphic to the direct sum of countably many $\mathbb{Z}/2\mathbb{Z}$'s.
Also, the nim-sum operation can be extended naturally on $\mathbb{Z}$ by using the 2's complement. 

\begin{definition}[minimum excluded number]
Let $T$ be a proper subset of $\mathbb{N}_{0}$. Then $\mathrm{mex}\ $$T$ is defined to be the least nonnegative integer not contained in $T$, namely
\begin{center}
$\mathrm{mex}\ T=\mathrm{min} (\mathbb{N}_0 \setminus T)$.
\end{center}
\end{definition}

\begin{definition}[Grundy value]
We denote the end position by $E$. Let $G$ be a game position.
The value $\mathcal{G}(G)$ is defined as follows:
\[
  \mathcal{G}(G)=\left \{ \begin{array}{cc}
    0 & (G=E) \\ 
   \mathrm{mex} \{\mathcal{G}(G')\mid G \rightarrow G'\} & (G \neq E).
  \end{array}\right.
\]
Moreover, $\mathcal{G}(G)$ is called the Grundy value of $G$.
\end{definition}

\begin{theorem}[Sprague\cite{Sprague}, Grundy\cite{Grundy}]
We have the following  for general short impartial games.
\begin{center}
$\mathcal{G}(G)\neq0$ $\Longleftrightarrow$ $G$ is an $\mathcal{N}$-position\\
$\mathcal{G}(G)=0$ $\Longleftrightarrow$ $G$ is a $\mathcal{P}$-position.
\end{center}
\end{theorem}

Therefore, we only need to decide the Grundy value of positions for winning strategy in impartial games.

Grundy value is also useful for analysis of disjunctive sum.

If $G$ and $H$ are any two game positions, the disjunctive sum of $G$ and $H$ (written as $G + H$) is defined as follows: each player must make a move in either $G$ or $H$ (but not both) on his turn. 
\begin{theorem}[Sprague-Grundy Theorem \cite{Sprague}]
Let $G$ and $H$ be two game positions. Then
\begin{center}
$\mathcal{G}(G+H)=\mathcal{G}(G)\oplus \mathcal{G}(H)$.
\end{center}
\end{theorem}

\begin{theorem}[Grundy\cite{Grundy}]
The Grundy value of Nim position $(m_1,\ldots,m_n)$ is the following:
\begin{center}
$\mathcal{G}(m_1,\ldots,m_n)=m_1\oplus \cdots \oplus m_n$.
\end{center}
\end{theorem}

\subsection{Welter's Game and the Welter function}
Welter's Game is an impartial game investigated by Welter in 1954. Since it was also investigated by Mikio Sato, it is often called Sato's Game in Japan. The rules of Welter's Games are as follows:

 \begin{center}\label{fig:1}  
 \begin{tabular}{|c|c|c|c|c|c|c|c|c|c|c}
 \hline
 0&1&$\bullet$&$\bullet$&4&$\bullet$&6&$\bullet$&8&9&$\cdots$  \\\hline
 \end{tabular}
 \end{center} 

\begin{itemize}
\item\ It is played with several coins placed on a belt divided into
  squares numbered with the nonnegative integers
  $0,1,2,\ldots$ from the left as shown in Fig.~\ref{fig:1}.

\item\ The legal move is to move any one coin from its present square
  to any unoccupied square with a smaller number.
\item\ The game terminates when a player is unable to move a coin,
  namely, the coins are jammed in squares with the smallest possible
  numbers as shown in Fig.~\ref{fig:2}.
\end{itemize}

 \begin{center}\label{fig:2}
 \begin{tabular}{|c|c|c|c|c|c|c|c|c|c|c}
 \hline
 $\bullet$&$\bullet$&$\bullet$&$\bullet$&4&5&6&7&8&9&$\cdots$  \\\hline
 \end{tabular}
 \end{center}

This game is equivalent to Nim with an additional rule that you are not allowed to make two heaps with the same number of tokens.

In what follows, when an expression includes both nim-sum and the four basic operations of arithmetic without parentheses, we will make it a rule to calculate nim-sums prior to the others, and we express the nim-summation by the symbol $\displaystyle \sum_{}^{\oplus}$.

\begin{lemma}[Conway\cite{Conway}]
For integer $n$,
\begin{center}
$n\oplus(-1)=-1-n$.
\end{center}
\end{lemma}

\begin{definition}[mating function]
Mating function $(x \mid y)$ is defined by
\[
  (x \mid y)=\left \{\begin{array}{cc}
  2^{n+1} -1 & (x\equiv y \pmod{2^n},\quad x\not\equiv y \pmod{2^{n+1}})\\
-1 & (x=y).
  \end{array}\right.
\]
Particularly, if $x$ and $y$ have different parities, then $(x \mid y)=1$.
\end{definition}

Then we have the following:
\begin{center}
$(x \mid y)=(x-y)\oplus(x-y-1)$, and $(x \mid y)=(x+a \mid y+a)=(x \oplus a \mid y \oplus a)$.
\end{center}

\begin{definition}[animating function]
For any nonnegative integers $a,b,c,d,\ldots$, a function $f(x)$ of form 
\begin{center}
$f(x)=(((x \oplus a )+b)\oplus c)+d \oplus \cdots$ 
\end{center}
is called an animating function.
\end{definition}

If $f$ and $g$ are animating functions, $f(g(x))$ and $f^{-1}(x)$
are clearly animating functions.
Also, we have $f^{-1}(x)=((((\cdots x \cdots)-d)\oplus)-b)\oplus a$.
Thus, the set of all animating functions forms a group with respect to composition.

\begin{definition}[Welter function]
Let $(a_1,\ldots,a_n)$ be a Welter's Game position. Then we define the value $[a_1|\cdots|a_n]$ of Welter function at $(a_1,\ldots,a_n)$ as follows:
\begin{center}
$[a_1|\cdots|a_n]=a_1\oplus \cdots \oplus a_n\oplus \displaystyle \sum_{1\leq i < j\leq n}^{\oplus} (a_i \mid a_j)$.
\end{center}
\end{definition}

In the case of one coin, clearly $[a_1]=a_1$.
In the case of two coins，
\begin{center}
$[a_1|a_2]=a_1 \oplus a_2 \oplus(a_1 \mid a_2)=a_1 \oplus a_2 -1$.
\end{center}

Let $(a_1,\ldots,a_n)$ be a position in Welter's Game and $a_i$, $a_j$ the pair with the largest mating function value $(a_i \mid a_j)$ (that is, $a_i$ and $a_j$ are congruent to each other modulo the highest possible power of 2 among all pairs). Then mating function values $(a_i \mid a_k)$ and $(a_j \mid a_k)$ cancel each other for all other $a_k$'s. 

\begin{theorem}[Conway\cite{Conway}]
When we mate pairs with the largest mating function value in order，we have the following equality．For Welter function of $n$ arguments
\[ [a_1|\cdots|a_n] = \left \{\begin{array}{cc}
[a_1|a_2] \oplus [a_3|a_4] \oplus \cdots & (n:\mathrm{even})\\

[a_1|a_2] \oplus [a_3|a_4] \oplus \cdots \oplus [a_n] & (n:\mathrm{odd}),
  \end{array}\right.
\]
where $[a_1|a_2],[a_3|a_4],\ldots$ is arranged in order of the values of mating function. 
\end{theorem}

By using this equality and formulas $[a_1]=a_1$ and $[a_1|a_2]=a_1\oplus a_2-1$, we can easily compute the value of Welter function．

\begin{lemma}[Conway\cite{Conway}]\label{parities} 
$a_1>a_1',a_2>a_2',a_3>a_3',\ldots$ are legal moves in Welter's Game, we have the following:
\begin{center}
$[a_1'|a_2|a_3|\cdots]=[a_1|a_2'|a_3|\cdots] \Longleftrightarrow[a_1'|a_2'|a_3|\cdots]=[a_1|a_2|a_3|\cdots]$.
\end{center}
\end{lemma}

\begin{theorem}[Conway\cite{Conway}]
Let $[a_1|\cdots|a_n]=s$ and let $s'$ be an integer. Welter function is an animating function with respect to each of its arguments, and an animating function is a bijection on $\mathbb{Z}$, so each of the equations
\begin{center}
$[a_1|\cdots|a_{i-1}|x|a_{i+1}|\cdots|a_n]=s'$ ($i=1,\cdots,n$) 
\end{center}
for the integers $x$ has a unique solution $x=a'_i$.
Moreover, if $s>s'$, then there is an index $i$ such that $a_i>a'_i$.
\end{theorem}

\begin{theorem}[Welter's Theorem\cite{Welter}]
The value of Welter function at each position in Welter's Game is equal to its Grundy value in Welter's Game. Namely, we have the following:
\begin{center}
$\mathcal{G}(a_1,\ldots,a_n)=[a_1|\ldots|a_n]$.
\end{center}
\end{theorem}

\section{Transfinite Game}
\subsection{Transfinite Nim}
First, we extend Nim into its transfinite version (Transfinite Nim) by allowing the size of the heaps of tokens to be a general ordinal number.
The legal move is to replace an arbitrary ordinal number $\alpha$ by a smaller number $\beta$.
Therefore, Transfinite Nim may not necessarily be short.

Let us denote by $\mathcal{ON}$ the class of all ordinal numbers.
Later we see that the nim-sum operation can be extended naturally on $\mathcal{ON}$. 

The following is known about general ordinal numbers.

\begin{theorem}[Cantor Normal Form theorem\cite{Jech}]
Every $\alpha \in \mathcal{ON}  (\alpha >0)$ can be expressed as
\begin{center}
$\alpha=\omega^{\gamma_k} \cdot m_{k}+\cdots+\omega^{\gamma_{1}} \cdot m_{ 1}+\omega^{\gamma_0} \cdot m_{0}$,
\end{center}
where $k$ is a nonnegative integer, $m_0,\ldots,m_k \in \mathbb{N}_0 \setminus\{0\}$, and $\alpha \geq \gamma_k>\cdots>\gamma_1>\gamma_0\geq0$.
\end{theorem}

Let $\alpha_1,\ldots,\alpha_n$ be ordinal numbers.
Then, each $\alpha_i$, $i=1,\ldots,n$ is expressed by using finite by many common powers $\gamma_0,\ldots,\gamma_k$ as: 
\begin{center}
$\alpha_i=\omega^{\gamma_k} \cdot m_{ik}+\cdots+\omega^{\gamma_{1}} \cdot m_{ i1}+\omega^{\gamma_0} \cdot m_{i0}$,
\end{center}
where $m_{ik}\in \mathbb{N}_0$.

Next, we will define the minimal excluded number of a set of ordinals and the Grundy value of a position in general Transfinite Game.

\begin{definition}[minimal excluded number]
Let $T$ be a proper subclass of $\mathcal{ON}$. Then $\mathrm{mex}\ $$T$ is defined to be the least ordinal number not contained in $T$, namely
\begin{center}
$\mathrm{mex}\ T=\mathrm{min} (\mathcal{ON}\setminus T)$.
\end{center}
\end{definition}

\begin{definition}[Grundy value]
Let $G$ be an impartial game (it may not necessarily be short) and  $E$ be the end position. The value $\mathcal{G}(G)$ is defined as 
\[
  \mathcal{G}(G)=\left \{ \begin{array}{cc}
    0 & (G=E) \\ 
   \mathrm{mex} \{\mathcal{G}(G')\mid G \rightarrow G'\} & (G \neq E).
  \end{array}\right.
\]
\end{definition}

\begin{theorem}
We have the following for Transfinite impartial games:
\begin{center}
$\mathcal{G}(G)\neq0$ $\Longleftrightarrow$ $G$ is an $\mathcal{N}$-position\\
$\mathcal{G}(G)=0$ $\Longleftrightarrow$ $G$ is a $\mathcal{P}$-position.
\end{center}
\end{theorem}

\begin{definition}
For ordinal numbers $\alpha_1,\ldots,\alpha_n \in \mathcal{ON}$, we define their nim-sum 
as follows:
\begin{center}
$\alpha_1 \oplus \cdots \oplus \alpha_n=\displaystyle \sum_{k}^{}  \omega^{\gamma_k}(m_{1k}\oplus \cdots \oplus m_{nk})$.
\end{center}
\end{definition}

\begin{theorem}
For Transfinite Nim position $(\alpha_1,\ldots,\alpha_n)\subseteq \mathcal{ON}^n$，we have the following:
\begin{center}
$\mathcal{G}(\alpha_1,\ldots,\alpha_n)=\alpha_1 \oplus \cdots \oplus \alpha_n$.
\end{center}
\end{theorem}
\begin{proof}
The proof is by induction.
Let $\alpha_1 \oplus \cdots \oplus \alpha_n=\alpha$ $(\alpha \in \mathcal{ON})$.
We have to show that, for each $\beta$ ($< \alpha$), there exists a position reached by a single move from $(\alpha_1,\ldots,\alpha_n)$ and that its Grundy value is $\beta$.

Let $(\alpha_1,\ldots,\alpha_n) \rightarrow (\beta_1,\ldots,\beta_n)$,
by  induction assumption we have
\begin{center}
$\mathcal{G}(\beta_1,\ldots,\beta_n)=\beta_1 \oplus \cdots \oplus \beta_n$.
\end{center}

If $\alpha=0$, no ordinal $\beta$ ($\beta < \alpha$) exists.
We can assume $\alpha > 0$.

We can write $\alpha$ and $\beta$ as
\begin{center}
$\alpha=\omega^{\gamma_k} \cdot a_k+\cdots+\omega^{\gamma_k} \cdot a_1+a_{0}$\\
$\beta= \omega^{\gamma_k} \cdot b_k+\cdots+\omega^{\gamma_k} \cdot b_1+b_{0}$,
\end{center}
where $a_0,\ldots,a_k, b_0,\ldots,b_k \in \mathbb{N}_0$.
By definition,
\begin{center}
$a_s=m_{1s}\oplus \cdots \oplus m_{ns}$, for $s=1,\ldots,k$.
\end{center}

Since $\alpha>\beta$, there exsists $s$ such that  
\begin{center}
$a_s > b_s$, $a_t=b_t$  for all $t$ $(< s)$.
\end{center}

As in the strategy of original Nim, since $a_s > b_s$, there is an index $i$ such that
\begin{center}
$m_{is} > m_{is}\oplus a_s \oplus b_s$.
\end{center}

We define 
\begin{center}
$m'_{it}=m_{it}\oplus a_s \oplus b_s$ for all $t$ $(\leq s)$
\end{center}
and
\begin{align*}
\alpha'_i&=\omega^{\gamma_k} \cdot m_{ik}+\cdots\omega^{\gamma_s+1} \cdot m_{is+1}+\omega^{\gamma_s} \cdot m'_{is}\\
&+\omega^{\gamma_s-1} \cdot m'_{is-1}+\cdots+\omega^{\gamma_0} \cdot m'_{i0},
\end{align*}
where $m_{is}\oplus a_s \oplus b_s=m'_{is}$.

If we put $\alpha'_i=\beta_i$, $\alpha_j=\beta_j$ $(j \neq i)$.
Then, $\alpha_i > \beta_i$ and  we have
\begin{center}
$\beta_1 \oplus \cdots \beta_{i-1}\oplus \beta_i\oplus \beta_{i+1}\oplus \cdots \beta_n=\beta$
\end{center}

Therefore, for each $\beta$ ($< \alpha$), there is a position $(\beta_1,\ldots,\beta_n)$ reached by a single move from $(\alpha_1,\ldots,\alpha_n)$.
\end{proof}

\begin{example}
In the case of position $(1,\omega\cdot 2+4,\omega^{2}\cdot 3+9,\omega^{2}\cdot 2+\omega \cdot
4+16,\omega^{2}+\omega \cdot 5+25)$:

Let us calculate the value of $\alpha_{1}\oplus\alpha_{2}\oplus\alpha_{3}\oplus\alpha_{4}\oplus\alpha_{5}$.

We get
\begin{align*}
\alpha_1&=\omega^{\beta_{2}}\cdot m_{12}+\omega^{\beta_{1}}\cdot m_{11}+m_{10}
=\omega^{2}\cdot 0+\omega\cdot 0+1\\
\alpha_2&=\omega^{\beta_{2}}\cdot m_{22}+\omega^{\beta_{1}}\cdot m_{21}+m_{20}
=\omega^{2}\cdot 0+\omega\cdot 2+4\\
\alpha_3&=\omega^{\beta_{2}}\cdot m_{32}+\omega^{\beta_{1}}\cdot m_{31}+m_{30}
=\omega^{2}\cdot 3+\omega\cdot 0+9\\
\alpha_4&=\omega^{\beta_{2}}\cdot m_{42}+\omega^{\beta_{1}}\cdot m_{41}+m_{40}
=\omega^{2}\cdot 2+\omega\cdot4+16\\
\alpha_5&=\omega^{\beta_{2}}\cdot m_{52}+\omega^{\beta_{1}}\cdot m_{51}+m_{50}
=\omega^{2}\cdot 1+\omega\cdot5+25.
\end{align*}
So, we have
\begin{eqnarray*}
m_{12}\oplus m_{22}\oplus m_{32}\oplus m_{42}\oplus m_{52} 
&=& 0\oplus 0\oplus 3\oplus  2\oplus  1\\
&=& 0\\
m_{11}\oplus m_{21}\oplus m_{31}\oplus m_{41}\oplus m_{51}
&=& 0\oplus 2\oplus 0\oplus 4\oplus 5\\
&=& 3\\
m_{10}\oplus m_{20}\oplus m_{30}\oplus m_{40}\oplus m_{50}
&=& 1\oplus 4\oplus 9\oplus 16\oplus 25\\
&=& 5.
\end{eqnarray*}
Thus, by the definition of nim-sum in general ordinal number 
\begin{center}
$\alpha_{1}\oplus\alpha_{2}\oplus\alpha_{3}\oplus\alpha_{4}\oplus\alpha_{5}
=\omega\cdot 3+5$.
\end{center}

Therefore, this position is an $\mathcal{N}$-position,
and the legal good move is $\omega\cdot 2+4$ $\rightarrow$ $\omega+1$.
\end{example}
\subsection{Transfinite Welter's Game}
In Transfinite version, the size of the belt of Welter's Game is
extended into general ordinal numbers, but played with finite number of coins. The legal move is to move one coin toward the left (jumping
is allowed), and you cannot place two or more coins on the same square
as in the original Welter's Game (see Fig.~\ref{fig:3}).
We will define Welter function of a position of Transfinite Welter's Game.

 \begin{center}\label{fig:3}
 \begin{tabular}{|c|c|c|c|c|c|c|c|c|c|c}
 \hline
 0&1&$\bullet$&3&$\cdots$&$\omega$&$\bullet$&$\omega+2$&$\cdots  $&$\omega^2$&$\cdots$\\\hline
 \end{tabular}
 \end{center}

\begin{definition}
Let $\alpha_1,\ldots,\alpha_n \in \mathcal{ON}$. Each $\alpha_i$ can be expressed as $\alpha_i=\omega \cdot \lambda_i+m_{i}$, where $\lambda_i \in \mathcal{ON}$ and $m_i \in \mathbb{N}_0$. Welter function in general ordinal numbers is defined as follows:
\begin{center}
$[\alpha_1|\cdots|\alpha_n]
=\omega\cdot(\lambda_1\oplus\cdots \oplus \lambda_n)+\displaystyle \sum_{\lambda \in \mathcal{ON}}^{\oplus}[S_{\lambda}]$,
\end{center}
where $[S_{\lambda}]$ is Welter function, and $S_{\lambda}=\{m_{n}\mid\lambda_n=\lambda \}$.
\end{definition}

We obtain the following main theorem.
\begin{theorem}
Let $\alpha_1,\ldots,\alpha_n \in \mathcal{ON}$. Grundy value of general position $(\alpha_1,\ldots,\alpha_n)$ in Transfinite Welter's Game is equal to its Welter function. Namely, we have the following:
\begin{center}
$\mathcal{G}(\alpha_1,\ldots,\alpha_n)=[\alpha_1|\cdots|\alpha_n]$.
\end{center}
\end{theorem}
\begin{proof}
Let $[\alpha_1|\cdots|\alpha_n]=\alpha$.
We have to show that, for each $\beta$ ($< \alpha$), there exists a position with Grundy value $\beta$ which is reached by a single move from $(\alpha_1,\ldots,\alpha_n)$.

Let $(\alpha_1,\ldots,\alpha_n) \rightarrow (\beta_1,\ldots,\beta_n)$. Then by the assumption of induction we have
\begin{center}
$\mathcal{G}(\beta_1,\ldots,\beta_n)=[\beta_1|\cdots|\beta_n]$.
\end{center}

If $\alpha=0$, there exist no $\beta$ $(< \alpha)$.
We can assume $\alpha > 0$ and
\begin{center}
$\alpha= \omega\cdot\lambda+a_{0}$ and $\beta= \omega\cdot\lambda'+b_{0}$,
\end{center}
where $\lambda$, $\lambda' \in \mathcal{ON}$, $a_0$, $b_0 \in \mathbb{N}_0$.
Since $\alpha>\beta$，we have
\begin{center}
$(\lambda > \lambda')$ or $(\lambda=\lambda'$ and $a_0 > b_0)$.
\end{center}

In the latter case, since $a_0 =\displaystyle \sum_{\lambda \in \mathcal{ON}}^{\oplus}[S_{\lambda}]> b_0$, from theory of Nim\cite{Bouton}\cite{Sprague}\cite{Grundy}, there exists some $\lambda_0$ and nonnegative integer $c_0$ $(<[S_{\lambda_0}])$ such that
\begin{center}
$a_0\oplus[S_{\lambda_0}]\oplus c_0=b_0$.
\end{center}

Next since $[S_{\lambda_0}]>c_0$, from theory of Welter function\cite{Conway}, there is an index $i$ and $m'_i$ $(<m_i)$ such that
$m_i \in S_{\lambda_0}$ and
$[S'_{\lambda_0}]=c_0$,
where $S'_{\lambda_0}$ is the set obtained from $S_{\lambda_0}$ by replacing $m_i$ with $m'_i$.
Thus, the move from $\alpha_i=\omega\cdot\lambda_i+m_i$ to $\alpha'_i=\omega\cdot\lambda_i+m'_i$ changes its Grundy value from $\alpha=\omega\cdot\lambda+a_0$ to $\beta=\omega\cdot\lambda+b_0$.

In the former case, as in Transfinite Nim, there is an index $i$ and $\lambda'_i$ $(< \lambda)$ such that 
$(\lambda_1,\ldots,\lambda_{i-1},\lambda'_i,\lambda_{i+1},\ldots,\lambda_n)$ 
has Grundy value $\lambda'$ and we can adjust the finite part of $\alpha_i$ so that the resulting Welter function to be $\beta$.

Therefore, for each $\beta$ ($< \alpha$), there is a position reached by a single move from $(\alpha_1,\ldots,\alpha_n)$ and its Grundy value is $\beta$.
\end{proof}

\begin{corollary}\label{transfinite}
A position in Transfinite Welter's Game is a $\mathcal{P}$-position if and only if it satisfies the following conditions:
\[\left \{\begin{array}{cc}
\omega\cdot(\lambda_1\oplus\cdots \oplus \lambda_n)=0\\
\displaystyle \sum_{\lambda \in \mathcal{ON}}^{\oplus}[S_{\lambda}]=0.
  \end{array}\right.\]
\end{corollary}

By this corollary, we can easily calculate a winning move in Transfinite Welter's Game by its Welter function.

\begin{example}
In the case of position $(1,\omega\cdot 2+4,\omega\cdot 2+9,\omega^{2}+\omega \cdot
4+16,\omega^{2}+\omega \cdot 5+25)$:

Let us calculate the value of $[\alpha_{1}|\alpha_{2}|\alpha_{3}|\alpha_{4}|\alpha_{5}]$.
We get
\begin{align*}
\alpha_1&=\omega^{\beta_{2}}\cdot m_{12}+\omega^{\beta_{1}}\cdot m_{11}+m_{10}
=\omega^{2}\cdot 0+\omega\cdot 0+1\\
\alpha_2&=\omega^{\beta_{2}}\cdot m_{22}+\omega^{\beta_{1}}\cdot m_{21}+m_{20}
=\omega^{2}\cdot 0+\omega\cdot 2+4\\
\alpha_3&=\omega^{\beta_{2}}\cdot m_{32}+\omega^{\beta_{1}}\cdot m_{31}+m_{30}
=\omega^{2}\cdot 0+\omega\cdot 2+9\\
\alpha_4&=\omega^{\beta_{2}}\cdot m_{42}+\omega^{\beta_{1}}\cdot m_{41}+m_{40}
=\omega^{2}\cdot 1+\omega\cdot4+16\\
\alpha_5&=\omega^{\beta_{2}}\cdot m_{52}+\omega^{\beta_{1}}\cdot m_{51}+m_{50}
=\omega^{2}\cdot 1+\omega\cdot5+25.
\end{align*}
So, we have
\begin{eqnarray*}
m_{12}\oplus m_{22}\oplus m_{32}\oplus m_{42}\oplus m_{52} 
&=& 0\oplus 0\oplus 0\oplus  1\oplus  1\\
&=& 0\\
m_{11}\oplus m_{21}\oplus m_{31}\oplus m_{41}\oplus m_{51}
&=& 0\oplus 2\oplus 2\oplus 4\oplus 5\\
&=& 1\\
\lbrack{m_{10}}\rbrack\oplus [m_{20} \mid m_{30}] \oplus [m_{40}]\oplus\bigl[m_{50}]
&=&[1]\oplus[4\mid9]\oplus[16]\oplus [25]\\
&=&1\oplus (4\oplus 9-1)\oplus 16\oplus 25\\
&=&4.
\end{eqnarray*}
Therefore, by the definition Welter function for general ordinal number 
\begin{center}
$[\alpha_{1}|\alpha_{2}|\alpha_{3}|\alpha_{4}|\alpha_{5}]=\omega+4$.
\end{center}

Since, this shows that we are in an $\mathcal{N}$-position, we will calculate a winning move.

First, we choose a move that satisfies the first condition of Corollary\ref{transfinite}.
Clearly we should not make a move that will change the coefficient of $\omega^{\beta_2}=\omega^{2}$. 
So we will choose a move that will change the coefficient of $\omega^{\beta_1}=\omega^{1}$ to be $0$.
The same strategy in Transfinite Nim, shows that
\begin{center}
$(2\oplus 2\oplus 4\oplus 5)\oplus 1=1\oplus 1=0$.
\end{center}

Thus, the only legal move is $5$ $\rightarrow$ $5\oplus 1=4$.
So, our good move is in $\omega\cdot 5+25$.
Then，in such moves，we will search for a move that satisfy the second condition．
It is obtained from the knowledge of Welter function.

The finite part should satisfy
\begin{center}
$1\oplus [4\mid9]\oplus [x\mid16]=0$.
\end{center}
So we have 
\begin{center}
$x=6$.
\end{center}

Therefore, the only good move is $\omega\cdot 5+25$ $\rightarrow$ $\omega\cdot 4+6$.

In fact，
\begin{eqnarray*}
m_{12}\oplus m_{22}\oplus m_{32}\oplus m_{42}\oplus m_{52} 
&=& 0\oplus 0\oplus 0\oplus  1\oplus  1\\
&=& 0\\
m_{11}\oplus m_{21}\oplus m_{31}\oplus m_{41}\oplus m_{51}
&=& 0\oplus 2\oplus 2\oplus 4\oplus 4\\
&=& 0\\
\lbrack{m_{10}}\rbrack \oplus [m_{20}\mid m_{30}]\oplus [m_{40}]\oplus[m_{50}]
&=&[1]\oplus[4\mid9]\oplus[6\mid16]\\
&=&1\oplus (4\oplus 9-1)\oplus (6\oplus 16-1)\\
&=&1\oplus 12\oplus 13\\
&=& 0.
\end{eqnarray*}

Thus, this position is a $\mathcal{P}$-position.

\end{example}

\addcontentsline{toc}{section}{{\bf References}}


\begin{thebibliography}{99}
\bibitem{Albert}
Albert, H. M., Nowakowski, J. R., Wolfe, D., {\it Lessons in play, an introduction to combinatorial Game theory}, A. K. Peters 2007.
\bibitem{Berlekamp}
Berlekamp, E. R., Conway, J. H., Guy, R. K., {\it Winning Ways for Your Mathematical Plays}，Vol 1-4, A. K.  Peters, 2001-2004．
\bibitem{Bouton}
Bouton, C. L., Nim, a game with a complete mathmatical theory, {\it Ann. of math.} {\bf 3} (1902), 35-39．
\bibitem{Conway}
Conway, J. H., {\it On Numbers And Games} (second edition)，A. K.  Peters, 2001．
\bibitem{Grundy}
Grundy, P. M., Mathematics and games, {\it Eureka}, {\bf 2} (1939), 6-8. 
\bibitem{Jech}
Jech, T., {\it Set theory} (third edition), Springer, 2002.
\bibitem{Siegel}
Siegel, A. N., {\it Combinatorial Game theory}, American Mathematical Society, 2013．
\bibitem{Sprague}
Sprague, R. P., Uber mathematische Kampfspiele, Tohoku Math. J., {\bf 41} (1935-6), 291-301. 
\bibitem{Welter}
Welter, C. P., The theory of a class of games on a sequence of squares, in terms of the advancing operation in a special abelian group, {\it Indagationes Math.}, {\bf 16} (1954), 194-200. 
\end{thebibliography}
\end{document}